\documentclass[a4paper, reqno, 11pt]{article}
\usepackage{amsmath, amsfonts, amsthm, amssymb}
\usepackage{graphicx}
\usepackage{float}
\usepackage{overpic}
\usepackage{verbatim}
\usepackage{color}
\usepackage[sort]{cite}

\numberwithin{equation}{section}

\renewcommand{\geq}{\geqslant}
\renewcommand{\leq}{\leqslant}

\DeclareMathOperator{\R}{\mathbb{R}}
\DeclareMathOperator{\Q}{\mathbb{Q}}
\DeclareMathOperator{\N}{\mathbb{N}}

\DeclareMathOperator{\cl}{cl}
\DeclareMathOperator{\fL}{\mathfrak{L}}
\DeclareMathOperator{\eps}{\varepsilon}

\DeclareMathOperator{\supp}{supp}

\DeclareMathOperator{\dimA}{\dim\sb{\mathrm{A}}}
\DeclareMathOperator{\dimL}{\dim\sb{\mathrm{L}}}

\DeclareMathOperator{\dimAt}{{\dim}_{\mathrm{A}}^{\theta}}
\DeclareMathOperator{\dimLt}{{\dim}_{\mathrm{L}}^{\theta}}

\DeclareMathOperator{\dimqA}{{\dim}_{\mathrm{qA}}}
\DeclareMathOperator{\dimqL}{{\dim}_{\mathrm{qL}}}

\allowdisplaybreaks

\newtheorem{theorem}{Theorem}[section]
\newtheorem{proposition}[theorem]{Proposition}

\newtheorem{lemma}[theorem]{Lemma}
\newtheorem{conjecture}[theorem]{Conjecture}
\newtheorem{question}[theorem]{Question}

\title{Regularity versus smoothness of measures}

\author{Jonathan M.~Fraser\footnote{JMF was financially supported by  the EPSRC Standard Grant EP/R015104/1 and the  Leverhulme Trust Research Project Grant RPG-2019-034.} 
\and Sascha Troscheit\footnote{ST was financially supported by the A\"OU Collaborative Grant 103\"ou6.}}

\begin{document}

\maketitle

\vspace{-5mm}
\begin{center}
	$^*$Mathematical Institute, University of St Andrews, North Haugh, St Andrews, Fife, KY16 9SS, UK.\\
	E-mail: jmf32@st-andrews.ac.uk
	\\ \vspace{3mm}
	$^\dagger$Faculty of Mathematics, University of Vienna, Oskar-Morgenstern-Platz 1, 1090 Wien, Austria.	\\
	E-mail: sascha.troscheit@univie.ac.at
	\end{center}

\begin{abstract}
The Assouad and lower dimensions and dimension spectra   quantify the \emph{regularity} of a measure
by considering the relative measure of concentric balls.  On the other hand, one can quantify the
\emph{smoothness} of an absolutely continuous measure by considering  the $L^p$ norms of its
density.  We establish sharp  relationships between these two notions. Roughly speaking,
we show that smooth measures must be regular, but that regular measures need not be smooth. \\


\end{abstract}

\section{Introduction and preliminaries.}

\subsection{Assouad type dimensions and spectra of measures.}
The Assouad and lower dimensions of  measures, also known as the regularity dimensions, are
important notions in dimension theory and geometric measure theory. They capture extremal scaling behaviour 
of measures by considering the relative measure of concentric balls and have a strong connection to
doubling properties. A fundamental result is that a measure has finite Assouad dimension if and only
if it is doubling
and that a measure has positive lower dimensions if and only if it is inverse doubling, see
e.g.~\cite{Kaenmaki17,Kaenmaki13}.
As such, these dimensions quantify the regularity of a measure. The Assouad and lower spectrum
provide a more nuanced analysis along these lines by fixing the relationship between the  radii of
the concentric balls according to a parameter $\theta \in (0,1)$ which is then varied to produce the
spectra.   Motivated by progress on  Assouad type dimensions and spectra for \emph{sets},  the
analogues for measure  were   investigated in \cite{Kaenmaki17,Kaenmaki13,Fraser18,
Hare18,Hare18-2}.

Throughout we  assume that $\mu$ is a Borel probability measure  on a compact metric space $(X,d)$.
These assumptions can be weakened in places but we  make them for expository reasons. We write
$\supp \mu$ for the support of $\mu$.  
The \emph{Assouad dimension} of $\mu$ is defined as
\[
\dimA \mu\ = \  \inf \Bigg\{ \alpha \  : \ ( \exists  \, C>0) \, (\forall \, 0<r<R<1) \, (\forall x
\in \supp \mu), \  \frac{\mu(B(x,R))}{\mu(B(x,r))} \ \leq \ C \bigg(\frac{R}{r}\bigg)^\alpha  \Bigg\}.
\]
Its dual, the \emph{lower dimension}, is defined analogously as
\[
\dimL \mu \ = \  \sup\Bigg\{ \alpha \  : \ ( \exists  \, C>0) \, (\forall \, 0<r<R<1) \, (\forall x
\in \supp \mu), \  \frac{\mu(B(x,R))}{\mu(B(x,r))} \ \geq \ C \bigg(\frac{R}{r}\bigg)^\alpha  \Bigg\}.
\]
The \emph{Assouad spectrum} is the function defined by
\begin{multline*}
\theta  \ \mapsto \ \dimAt \mu \ = \  \inf \bigg\{ \alpha \  : \   (\exists C>0) \, (\forall 0<R<1)
\,  (\forall x \in \supp\mu) ,\\  \frac{\mu(B(x,R))}{\mu(B(x,R^{1/\theta}))} \ \leq \ C
\left(\frac{R}{R^{1/\theta}}\right)^\alpha \bigg\},
\end{multline*}
where  $\theta$ varies over  $(0,1)$.   The related \emph{quasi-Assouad dimension} can be defined by
\[
\dimqA \mu \ = \  \lim_{\theta \to 1} \dimAt \mu 
\]
when it is finite. This is not the original definition, which stems from  \cite{LuXi}, but a
convenient equivalent formulation which was established in \cite[Proposition 6.2]{Hare18-2},
following \cite{ScottishCanadian}.   Similarly, the  \emph{lower spectrum} is defined by
\begin{multline*}
\theta  \ \mapsto \ \dimLt \mu \ = \  \sup \bigg\{ \alpha \  : \   (\exists C>0) \, (\forall 0<R<1)
\,  (\forall x \in \supp\mu) ,\\  \frac{\mu(B(x,R))}{\mu(B(x,R^{1/\theta}))} \ \geq \ C
\left(\frac{R}{R^{1/\theta}}\right)^\alpha \bigg\}
\end{multline*}
and the \emph{quasi-lower dimension} by
\[
\dimqL \mu \ = \  \lim_{\theta \to 1} \dimLt \mu .
\]
The Assouad and lower spectra  (of sets) were defined in \cite{Spectraa} and are designed to extract
finer geometric information than the Assouad, lower, and box-counting dimensions considered in isolation.  
See the survey \cite{FraserSurvey} for more on this approach to dimension theory.

These Assouad-type dimensions and spectra are  related by
\begin{align*}
  \dimL \mu \leq \dimqL\mu \leq \dimLt \mu &\leq \inf_{x\in\supp\mu}\underline{\dim}_{\mathrm{loc}}\mu(x) \\
&\leq\sup_{x\in\supp\mu}\overline{\dim}_{\mathrm{loc}}\mu(x) \leq \dimAt\mu \leq \dimqA\mu \leq
\dimA\mu,
\end{align*}
where $\overline{\dim}_{\mathrm{loc}}\mu(x)$ and $\underline{\dim}_{\mathrm{loc}}\mu(x)$ are the
upper and lower 
local dimensions of $\mu$  at a point $x$.  We do not use the local dimensions but mention them here
to emphasise that the Assouad and lower dimensions are extremal, since most familiar notions of
dimensions for measures lie in between the infimal and supremal lower dimensions, e.g. the Hausdorff
dimension. For more information, including basic properties, concerning Assouad-type  dimensions of
measures, see \cite{Fraser18, Hare18,Hare18-2, Kaenmaki17,Kaenmaki13}.

\subsection{$L^p$ properties of measures.}

A probability measure $\mu$ supported on a compact subset $X \subset \mathbb{R}^d$ is \emph{absolutely
continuous} (with respect to the Lebesgue measure), if all Lebesgue null sets are given zero $\mu$
measure.  In particular, this means that there is a Lebesgue integrable function $f$, the
\emph{density} or \emph{Radon-Nikodym derivative}, such that 
\[
\mu(E) = \int_E f(x) \, dx
\]
for all Borel sets $E$.  Given $p\geq 1$, the  space $L^p(X)$ is defined to consist of all integrable functions $g$ such that
\[
\|g \|_p : = \left(\int_X g(x)^p\, dx\right)^{1/p} < \infty.
\]
The space $L^\infty(X)$ denotes the space of essentially bounded functions. In a slight abuse of
notation, we say $\mu \in L^p(X)$ if $\mu$ is absolutely continuous with density $f \in L^p$.  
Given an absolutely continuous measure $\mu$, we can thus understand how smooth $\mu$ is by
determining precisely  for which $p$ we have  $\mu \in L^p(X)$.  Since we assume $\mu$ is compactly
supported, $\mu \in L^{p_2}(X)$ implies $\mu \in L^{p_1}(X)$ for all $1 \leq p_1 \leq p_2 \leq
\infty$ and therefore it is harder for the $\mu$  to be in $L^p(X)$ as $p$ increases.  We think of
measures  being smoother if they lie in  $L^p(X)$ for larger $p$ and $L^\infty(X)$ as consisting of the smoothest measures possible, according to this analysis. 

One can consider absolute continuity with respect to arbitrary reference measures in place of the
Lebesgue measure.  Much of our work would also apply in this setting, but  we  focus on
$X=[a,b]\subset \R$ with the Lebesgue measure as the reference measure
and write $L^p$ instead of $L^p([a,b])$.  Our results also easily extend to higher dimensions, that
is, when the reference measure is $d$-dimensional Lebesgue measure.   We focus on the 1-dimensional
case with Lebesgue measure as the reference measure since this is the most natural and important
case and also to simplify our exposition.

It will also be useful to consider `inverse $L^p$ spaces'.  We write $f\in
L^{-p}$  if  the set $E =\{ x \in X : f(x) = 0\} \subset X$  is of Lebesgue measure zero and
\begin{equation*}
  \left(\int_{X\setminus E}1/f(x)^p dx\right)^{1/p} <\infty.
\end{equation*}
Analogous to above we write $\mu \in L^{-p}$ if $\mu$ is absolutely continuous with density in $  L^{-p}$.

\section{Main results: smoothness versus regularity.}

The objective of this article is to investigate the relationships between regularity and smoothness,
as described by Assouad type dimensions and spectra and $L^p$ properties, respectively.

First, we remark that for an
absolutely continuous measure $\mu$, the condition that $\mu \in L^p$ does not guarantee that
$\dimL\mu >0$ or  $\dimA\mu <\infty$.  All one can conclude is that $0\leq \dimL\mu\leq 1\leq
\dimA\mu\leq \infty$. 
Even the strong assumption that the density of a measure is bounded, does not guarantee a measure is doubling,
take for instance the density on $[-1,1]$ defined by
\[
  f(x)=\begin{cases}2^{-k} &\text{ if } 2^{-k} \leq x < 2^{-k+1}\\\frac{2}{3} &\text{ otherwise}\end{cases}
\]
for $k\in\N$. One can check that $\int_{-1}^1 f(x)dx =1$ and thus $\mu$ is a probability measure.
The ball $B(2^{-k},2^{-k})$ has measure $4/3 \cdot 4^{-k}$, whereas $\mu(B(2^{-k},2^{-k+1}))=2/3
\cdot2^{-k}+4/3 \cdot 4^{-k+1}$ and therefore
\[
\frac{\mu(B(2^{-k},2^{-k+1}))}{\mu(B(2^{-k},2^{-k}))} \ \geq \ 2^{k-1} \to \infty.
\]
Therefore, the measure is not doubling and, in particular,  $\dimA \mu=\infty$.  Bounded density is
also not enough to say something about the quasi-Assouad dimension or Assouad spectrum, see below.
It turns out we need to be able to control the density from both sides in order to get good
estimates for the Assouad type dimensions. Our main result establishes a sharp correspondence along
these lines.
\begin{theorem}
  \label{thm:main}
 Suppose $p_1, p_2 \in [1,\infty]$ are such that $\mu\in L^{p_1} \cap
  L^{-p_2}$.  If $p_1, p_2 < \infty$, then
  \begin{equation}\label{eq:main}
    \dimAt \mu \leq 1+\frac{p_1 + \theta p_2}{p_1 p_2 (1-\theta)} 
    \quad \text{and} \quad 
    \dimLt \mu \geq 1-\frac{\theta p_1 + p_2}{p_1 p_2 (1-\theta)} .
  \end{equation}
If  $\mu\in L^{\infty} \cap
  L^{-\infty}$, then $\mu$ is 1-Ahlfors regular and $\dimA \mu = \dimL \mu = 1$ and if $\mu\in L^{p_1} \cap
  L^{-\infty}$ or $\mu\in L^{\infty} \cap
  L^{-p_2}$, then  one can obtain bounds by taking the limit as $p_1$ or $p_2$ tends to infinity in
  \eqref{eq:main}.  Moreover, all of these bounds are sharp.  
\end{theorem}

 The fact that these bounds above are sharp shows that knowledge of
  $L^p$-smoothness and inverse $L^p$-smoothness are  not sufficient to give bounds on the regularity
  as measured by the quasi-Assouad and Assouad dimensions, or the quasi-lower and lower dimensions.
  This is seen by letting $\theta \to 1$.

\subsection{Proof of Theorem \ref{thm:main}.}

Throughout the rest of the paper we write $A \lesssim B$ to mean there exists a uniform constant
$c>0$ such that $A \leq c B$.   Similarly, we write $A \gtrsim B$ to mean $B \lesssim A$ and  $A
\approx B$ if $A \lesssim B$ and $A \gtrsim B$.

\subsubsection{Establishing the bounds.}

  The proof uses H\"older's inequality and the reverse H\"older inequality. That is, for all $p>1$
  and $q$ such that $1/p+1/q=1$ and measurable functions $f,g$ we have 
  \[ \lVert f g \rVert_1 \leq \lVert f \rVert_p \lVert g \rVert_q\quad \text{and} \quad
  \lVert f g \rVert_1 \geq \lVert f\rVert _{-p} \lVert g \rVert_{p/(p+1)},
  \]
  where for the latter we also require that  $f(x)\neq 0$ for almost   every $x$.  We note that
  $\lVert f\rVert _{-p} $ and $ \lVert g \rVert_{p/(p+1)}$ are not norms but convenient notation for
\[
\lVert f\rVert _{-p} = \left( \int f(x)^{-p} dx \right)^{-1/p}
\quad\text{and}\qquad
\lVert g \rVert_{p/(p+1)}= \left( \int f(x)^{p/(p+1)} dx \right)^{(p+1)/p},
\]
respectively.   We use the above inequalities to estimate 
  \begin{equation}\label{eq:toEstimate}
    \frac{\mu(B(x,R))}{\mu(B(x,r))} = \frac{\lVert f \cdot \chi_{B(x,R)}\rVert_1}{\lVert f \cdot
    \chi_{B(x,R)}\rVert_1},
  \end{equation}
  where $f$ is the density of $\mu$ and $\chi_A$ is the indicator function associated with a set $A$.

 Fix $\theta\in(0,1)$ and let $r=R^{1/\theta}$.  Write $q_1\in (1,\infty)$ for the H\"older
 conjugate of $p_1$, that is the unique value satisfying $1/p_1+1/q_1 = 1$.    Noting that
 $\frac{\lVert f\rVert_{p_1}}{\lVert f \rVert_{-p_2}}  \in (0,\infty)$ is a constant independent of
 $R$, we can bound  \eqref{eq:toEstimate} from above by
  \begin{align*}
    \frac{\mu(B(x,R))}{\mu(B(x,r))} \leq \frac{\lVert f
    \rVert_{p_1}\lVert\chi_{B(x,R)}\rVert_{q_1}}{\lVert f \rVert_{-p_2}\lVert
    \chi_{B(x,r)}\rVert_{p_2/(1+p_2)}} 
    &\lesssim     \frac{\left(\int \chi_{B(x,R)}^{q_1}d\mu\right)^{1/q_1}}{\left(\int
    \chi_{B(x,r)}^{p_2/(1+p_2)}d\mu\right)^{1+1/p_2}}\\
    &\lesssim  \dfrac{R^{1-1/p_1}}{r^{1+1/p_2}} \\
    &=   \left(\frac{R}{r}\right)^{\frac{1-1/p_1-1/\theta(1+1/p_2)}{1-1/\theta}}.
  \end{align*}
Therefore,
\begin{equation*}
  \dimAt\mu \leq \frac{1-1/p_1-1/\theta(1+1/p_2)}{1-1/\theta}
  =1+\frac{p_1+\theta p_2}{p_1 p_2 (1-\theta)},
\end{equation*}
as required.

We can bound  \eqref{eq:toEstimate} from below similarly  by
  \begin{align*}
    \frac{\mu(B(x,R))}{\mu(B(x,r))} \geq \frac{\lVert f \rVert_{-p_2}\lVert
    \chi_{B(x,R)}\rVert_{p_2/(1+p_2)}}{\lVert f
    \rVert_{p_1}\lVert\chi_{B(x,r)}\rVert_{q_1}}
    &\gtrsim  \dfrac{R^{1+1/p_2}}{r^{1-1/p_1}}\\
    &=   \left(\frac{R}{r}\right)^{\frac{1+1/p_2-1/\theta(1-1/p_1)}{1-1/\theta}}.
  \end{align*}
Therefore,
\begin{equation*}
  \dimLt\mu \geq \frac{1+1/p_2-1/\theta(1-1/p_1)}{1-1/\theta}
  =1-\frac{\theta p_1+ p_2}{p_1 p_2 (1-\theta)},
\end{equation*}
as required.

Finally, note that if  $\mu\in L^{\infty} \cap
  L^{-\infty}$, then
\[
\|f\|_{-\infty}\cdot r \leq \mu(B(x,r)) \leq \|f\|_\infty\cdot r
\]
for all $x$ in the support of $\mu$ and all $r \in (0,1)$. Therefore $\mu$ is $1$-Ahlfors regular.
The fact that the estimates are sharp is proved  in the following  subsections.

\subsubsection{Sharpness for the Assouad spectrum.} \label{asssharp}

The following lemma  shows that  the estimate for the Assouad spectrum in Theorem \ref{thm:main} is
sharp for  all $\theta \in (0,1)$.  Moreover, this shows that a measure can belong to $L^{p_1} \cap
L^{-p_2}$ whilst being non-doubling (that is, $\dimA \mu = \infty$) and even have infinite
quasi-Assouad dimension.
\begin{lemma} \label{sharpp1}
The bounds on the Assouad spectrum in Theorem \ref{thm:main} are sharp.  That is, given $p_1,p_2 \in
(1,\infty]$ there exists a probability measure  $\mu$ such that $\mu \in L^{p_1'} \cap L^{-p_2'}$
for all   $p_1'<p_1$ and $p_2'<p_2$, and
\begin{equation*}
  \dimAt\mu    =1+\frac{p_1+\theta p_2}{p_1 p_2 (1-\theta)}
\end{equation*}
for all $\theta \in (0,1)$.
\end{lemma}
\begin{proof}
Let $p_1,p_2 \in (1,\infty]$ and let $\mu$ be the probability measure supported on $[-1,1]$ with density 
\[
f(x)=\begin{cases}Cx^{-1/p_1}&0<x\leq 1\\ C(-x)^{1/p_2} &-1\leq x \leq 0\end{cases},
\]
 where  $C$ is chosen such that $\int f(x)dx = 1$, see Figure \ref{fig:monoexample1}.  We adopt the
 natural convention that
 $1/\infty=0$.  It is easily checked that $f\in L^{p_1'} \cap L^{-p_2'}$ for $p_1'<p_1$ and
 $p_2'<p_2$, but that  $ f\notin L^{p_1''} \cap L^{-p_2''}$ if either $p_1''=p_1$ or $p_2''=p_2$.

In order to bound the Assouad  spectrum from below it suffices to find points such that the relative
measure of balls centred at that point  is large.  To this end, let $R \in (0,1)$, $r=R^{1/\theta}$
and consider the point $-r$, where we find
 \begin{align*}
  \frac{\mu(B(-r,R))}{\mu(B(-r,r))} = \frac{\int_{-r-R}^{0}(-x)^{1/p_2} dx + \int_{0}^{R-r}x^{-1/p_1} dx}{\int_{-2r}^{0}(-x)^{1/p_2} dx} &\approx  \frac{(R+r)^{1+1/p_2}   + (R-r)^{1-1/p_1}}{r^{1+1/p_2} } \\
&  \gtrsim  \frac{R^{1-1/p_1}}{r^{1+1/p_2}}\\
& \gtrsim  \left(\frac{R}{r}\right)^{\frac{1-1/p_1-1/\theta(1+1/p_2)}{1-1/\theta}}.
\end{align*}
This shows that
\begin{equation*}
  \dimAt\mu \geq \frac{1-1/p_1-1/\theta(1+1/p_2)}{1-1/\theta}
  =1+\frac{p_1+\theta p_2}{p_1 p_2 (1-\theta)}
\end{equation*}
and
\[
\dimqA \mu = \dimA \mu = \infty.
\]
In fact, applying Theorem \ref{thm:main} to this example we see that
\begin{equation*}
  \dimAt\mu    =1+\frac{p_1+\theta p_2}{p_1 p_2 (1-\theta)}
\end{equation*}
for all $\theta \in (0,1)$.
\end{proof}

\begin{figure}[ht]
  \begin{center}\includegraphics[width=0.9\textwidth]{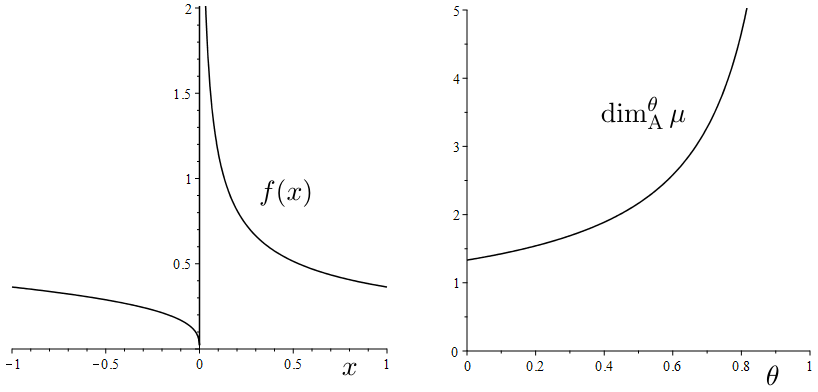}\end{center}
  \caption{A sharp example for the Assouad spectrum, with $p_1=2$,  $p_2 = 3$  and $C=4/11$.  The density is plotted on the left and the Assouad spectrum is plotted on the right.}
  \label{fig:monoexample1}
\end{figure}

\subsubsection{Sharpness for the lower spectrum.} \label{sharpsec2}

The following lemma  shows that  the estimate for the lower spectrum in Theorem \ref{thm:main} is
sharp for  all $\theta \in (0,1)$.  Moreover, this shows that a measure can belong to $L^{p_1} \cap
L^{-p_2}$ whilst being non-inverse doubling (that is, $\dimL \mu = 0$) and even have  quasi-lower
dimension equal to 0.
\begin{lemma} \label{sharpp2}
The bounds on the lower spectrum in Theorem \ref{thm:main} are sharp.  That is, given $p_1,p_2 \in
(1,\infty]$ there exists a probability measure  $\mu$ such that $\mu \in L^{p_1'} \cap L^{-p_2'}$
for all   $p_1'<p_1$ and $p_2'<p_2$, and 
\begin{equation*}
\dimLt \mu = \max\left\{ 1-\frac{\theta p_1+  p_2}{p_1 p_2 (1-\theta)}, \ 0 \right\}
\end{equation*}
for all $\theta \in (0,1)$.
\end{lemma}
\begin{proof}
Let $p_1,p_2 \in (1,\infty]$ and $\theta_i$ ($i \geq 1$) be an enumeration of $\Q\cap(0,1)$ such
that for every rational $q\in\Q\cap(0,1)$ there are
infinitely many $i\in\N$ such that $\theta_i = q$. Let $x_i = 2^{-i}+2^{-(i+1)}$ and $\mu$ be the
probability measure supported on $[0,1]$ with density 
\[
f(x)=\begin{cases}C2^{i/(\theta_ip_1)}& x \in B(x_i, 2^{-(i+1)/\theta_i})  \\ 
C2^{-i/p_2} &x \in B(x_i, 2^{-(i+1)} ) \setminus B(x_i, 2^{-(i+1)/\theta_i})  \\
0 & \textup{otherwise} \end{cases},
\]
 where  $C$ is chosen such that $\int f(x)dx = 1$ and the balls are assumed to be open, see Figure
 \ref{fig:monoexample2}.  We adopt the natural convention that $1/\infty=0$.  Moreover, by
 construction, the balls $B(x_i, 2^{-(i+1)} )$ are pairwise disjoint subsets of $[0,1]$ and so $f$ is
 well-defined.

There exists a constant $c>0$ such that $c^{-1}x^{1/p_2} \leq f(x)$ and therefore  $\mu \in
L^{-p_2'}$ for  $p_2'<p_2$. Moreover, for $p_1'<p_1$ we have
\[
\int f(x)^{p_1'} dx \lesssim \sum_i 2^{ip_1'/(\theta_ip_1)}2^{-i/\theta_i} \leq \sum_i 2^{i(p_1'/p_1-1)} < \infty
\]
and therefore  $\mu \in  L^{p_1'}$.

In order to bound the lower  spectrum from above it suffices  to find points such that the relative
measure of balls centred at that point  is small.  To this end, fix $\theta \in \Q \cap (0,1)$ and a
subsequence of the $\theta_i$, which we denote by  $\theta_k$, such that $\theta_k = \theta$ for all
$k$.  Along this sequence, let $R_k = 2^{-(k+1)}$, $r_k=R_k^{1/\theta}$ and consider the points
$x_k=2^{-k}+2^{-(k+1)}$, where we find
\begin{equation} \label{lsharp}
  \frac{\mu(B(x_k,R_k))}{\mu(B(x_k,r_k))} \lesssim  \frac{r_k 2^{k/(\theta p_1)}+ R_k2^{-k/p_2}}{r_k 2^{k/(\theta p_1)}}  \approx \frac{R_k^{1/\theta-1/(\theta p_1)}+ R_k^{1+1/p_2}}{R_k^{1/\theta-1/(\theta p_1)}}.
\end{equation}
Provided $1/\theta-1/(\theta p_1)> 1+1/p_2$ this gives an upper bound of
\[
\lesssim  \frac{R_k^{1+1/p_2}}{R_k^{1/\theta-1/(\theta p_1)}} = \left(\frac{R_k}{r_k} \right)^{\frac{1+1/p_2-1/\theta+1/(\theta p_1)}{1-1/\theta}}
\]
for \eqref{lsharp}. This shows
\[
\dimLt \mu \leq \frac{1+1/p_2-1/\theta+1/(\theta p_1)}{1-1/\theta} =  1-\frac{\theta p_1+  p_2}{p_1 p_2 (1-\theta)}.
\]
On the other hand, if $1/\theta-1/(\theta p_1)\leq 1+1/p_2$, this gives an upper bound of $\lesssim
1$ for \eqref{lsharp}, which implies $ \dimL \mu \leq 0$.  Putting these two cases together we get
\[
\dimLt \mu \leq \max\left\{ 1-\frac{\theta p_1+  p_2}{p_1 p_2 (1-\theta)}, \ 0 \right\}
\]
for all rational $\theta \in (0,1)$.  Since the lower spectrum is continuous in $\theta \in (0,1)$,
we therefore conclude this upper bound for all $\theta \in (0,1)$.  Moreover, we get
\[
\dimqL \mu = \dimL \mu = 0.
\]
In fact, applying Theorem \ref{thm:main} we get
\[
\dimLt \mu = \max\left\{ 1-\frac{\theta p_1+  p_2}{p_1 p_2 (1-\theta)}, \ 0 \right\}
\]
for all $\theta \in (0,1)$.\end{proof}

\begin{figure}[ht]
  \begin{center}\includegraphics[width=0.9\textwidth]{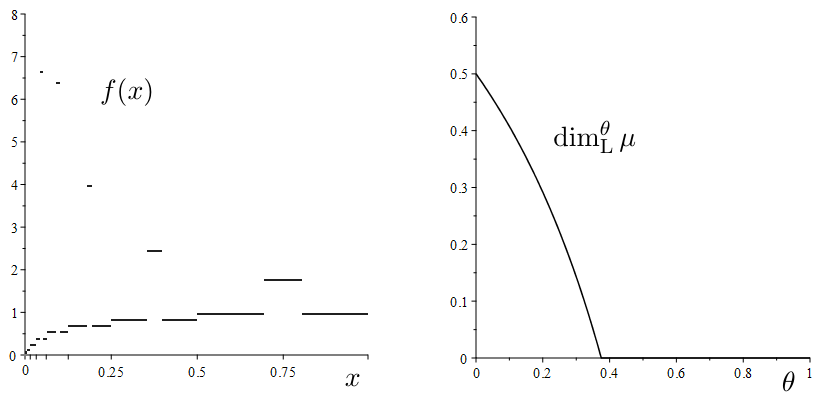}\end{center}
  \caption{A sharp example for the lower spectrum, with $p_1=2$ and $p_2 = 3$.  The density is plotted on the left and the Assouad spectrum is plotted on the right.  The density is not drawn to scale and is mostly for illustrative purposes. }
  \label{fig:monoexample2}
\end{figure}

\section{Piecewise monotonic densities.}

Given that the Assouad and lower spectra are dual notions, it is quite striking how different the
examples in the previous section are.  In particular, the measure exhibiting sharpness of the
Assouad spectrum bound in  Theorem \ref{thm:main} has a  piecewise monotonic density (Section
\ref{asssharp} and Figure \ref{fig:monoexample1}), whereas the  measure exhibiting sharpness of the
lower spectrum bound does not, and is rather more complicated to construct (Section \ref{sharpsec2}
and Figure \ref{fig:monoexample2}).  This turns out to be no coincidence.  Here, and in what
follows,  piecewise  means with  \emph{finitely} many pieces.

\begin{theorem}
  \label{thm:mono}
 Suppose $p_1, p_2 \in [1,\infty)$ are such that $\mu\in L^{p_1} \cap
  L^{-p_2}$.  If    $\mu$ has a  monotonic density, then
\[
    \dimAt \mu \leq \max \left\{ 1+\frac{1}{p_2 (1-\theta)} , \ 1+\frac{ \theta}{p_1  (1-\theta)}  \right\}
\]
and
\[
    \dimLt \mu \geq \min\left\{ 1-\frac{\theta }{p_2 (1-\theta)} , \ 1-\frac{1}{p_1  (1-\theta)}\right\} .
\]
 If    $\mu$ has a piecewise monotonic density, then
\[
  \dimAt \mu \leq 1+\frac{p_1 + \theta p_2}{p_1 p_2 (1-\theta)} 
\]
and
\[
    \dimLt \mu \geq \min\left\{ 1-\frac{\theta }{p_2 (1-\theta)} , \ 1-\frac{1}{p_1  (1-\theta)}\right\}.
\]
Moreover, all of these bounds are sharp.
\end{theorem}

\subsection{Proof of Theorem \ref{thm:mono}.}

\subsubsection{Establishing the bounds.}

We begin with the case when $\mu$ has a monotonic density, which we denote by $f$.   It follows that for all sufficiently small $R>0$ (depending on $f$) and all $x$, at least one of the following is satisfied:
\begin{enumerate}
\item  $f$ is  bounded above  by 2 on $B(x,R)$ 
\item  $f$ is  bounded below  by 2 on $B(x,R)$ 
\item  $f$ is bounded above by 3 and below by 1 on $B(x,R)$ 
\end{enumerate}
In each of these cases we follow the proof of Theorem \ref{thm:main} but we can obtain better estimates.  In Case 1 we have
\[
    \frac{\mu(B(x,R))}{\mu(B(x,r))} \leq \frac{2\lVert\chi_{B(x,R)}\rVert_{1}}{\lVert f \rVert_{-p_2}\lVert
    \chi_{B(x,r)}\rVert_{p_2/(1+p_2)}} 
        \lesssim  \frac{R}{r^{1+1/p_2}}  =   \left(\frac{R}{r}\right)^{\frac{1-1/\theta(1+1/p_2)}{1-1/\theta}},
\]
in Case 2 we have
\[
     \frac{\mu(B(x,R))}{\mu(B(x,r))} \leq \frac{\lVert f
    \rVert_{p_1}\lVert\chi_{B(x,R)}\rVert_{q_1}}{(1/2)\lVert
    \chi_{B(x,r)}\rVert_{1}} 
    \lesssim  \dfrac{R^{1-1/p_1}}{r} \\
    =   \left(\frac{R}{r}\right)^{\frac{1-1/p_1-1/\theta}{1-1/\theta}},
\]
and in Case 3
\[
    \frac{\mu(B(x,R))}{\mu(B(x,r))} \leq \frac{3\lVert\chi_{B(x,R)}\rVert_{1}}{ \lVert
    \chi_{B(x,r)}\rVert_{1}} 
        \lesssim  \frac{R}{r}.
\]
The estimate in the theorem follows. The lower spectrum case is similar and omitted.

We now consider  the case when $\mu$ has a piecewise monotonic density, which we denote by $f$. This is similar to the monotonic case, but an interesting phenomenon happens allowing us to improve  the general estimate for the lower spectrum in a way we cannot for the Assouad spectrum.

The upper bound for the Assouad spectrum is provided by Theorem \ref{thm:main} and the fact that this is sharp is shown by the example in Section \ref{asssharp}.  Therefore we may consider only the lower spectrum.  Since $f$ is piecewise monotonic it follows  that for all sufficiently small $R>0$ (depending on $f$) and all $x$,    at least one of the following is satisfied:
\begin{enumerate}
\item  $f$ is  bounded above  by 2 on $B(x,R)$ 
\item  $f$ is  bounded below  by 2 on $B(x,R)$ 
\item  $f$ is bounded above by 3 and below by 1 on $B(x,R)$ 
\item $B(x,R)$ can be written as a disjoint union of two intervals $I_1 \cup I_2$ such that $f$ is bounded above by 1 on $I_1$ and below by $1$ on $I_2$.  One of the intervals $I_1$ or $I_2$ may be empty and they can be closed, open or half open.  
\end{enumerate}
In each of these cases we follow the proof of Theorem \ref{thm:main} but we can obtain better estimates.  Cases 1-3 are covered above.  In Case 4, either $B(x,r) \subseteq I_1$ or not.  If $B(x,r) \subseteq I_1$, then
\[
    \frac{\mu(B(x,R))}{\mu(B(x,r))} \geq \frac{\lVert f \rVert_{-p_2}\lVert
    \chi_{B(x,R)}\rVert_{p_2/(1+p_2)}}{ \lVert\chi_{B(x,r)}\rVert_{1}}
    \gtrsim   \dfrac{R^{1+1/p_2}}{r} 
    =  \left(\frac{R}{r}\right)^{\frac{1+1/p_2-1/\theta}{1-1/\theta}}.
\]
  If $B(x,r)$ is not completely contained inside  $I_1$, then there must be an interval of length $R-r \gtrsim R$ contained in $I_2$ in which case
\[
    \frac{\mu(B(x,R))}{\mu(B(x,r))} \geq \frac{\mu(I_2)}{\mu(B(x,r))} \geq \frac{ \lVert
    \chi_{I_2}\rVert_{1}}{\lVert f
    \rVert_{p_1}\lVert\chi_{B(x,r)}\rVert_{q_1}}
    \gtrsim    \dfrac{R }{r^{1-1/p_1}}
     =   \left(\frac{R}{r}\right)^{\frac{1 -1/\theta(1-1/p_1)}{1-1/\theta}}.
\]
The estimate in the theorem follows. 

\subsubsection{Sharpness.}

It remains to show that the estimates in Theorem \ref{thm:mono} are sharp.  This  requires only one
further example, where $\mu$ is the measure on $[0,2]$ with density
\[
f(x)=\begin{cases}Cx^{-1/p_1}&0<x\leq 1 \\
 C(2-x)^{1/p_2} &1< x \leq 2\end{cases},
\]
where  $C$ is chosen such that $\int f(x)dx = 1$.  Minor adaptations of the above arguments yield
\[
    \dimAt \mu = \max \left\{ 1+\frac{1}{p_2 (1-\theta)} , \ 1+\frac{ \theta}{p_1  (1-\theta)}  \right\}
\]
and
\[
    \dimLt \mu = \max\left\{ \min\left\{ 1-\frac{\theta }{p_2 (1-\theta)} , \ 1-\frac{1}{p_1
    (1-\theta)}\right\},  \  0  \right\} .
\]
as required.

Note that for this family of sharp examples, the Assouad spectrum will exhibit a phase transition at
$\theta = p_1/p_2$ provided $p_1<p_2$ and the lower spectrum is constantly equal to 0 for
\[
\theta>\min\left\{\frac{p_2}{1+p_2}, \frac{p_1-1}{p_1}\right\}.
\]

\begin{figure}[ht]
  \begin{center}\includegraphics[width=0.9\textwidth]{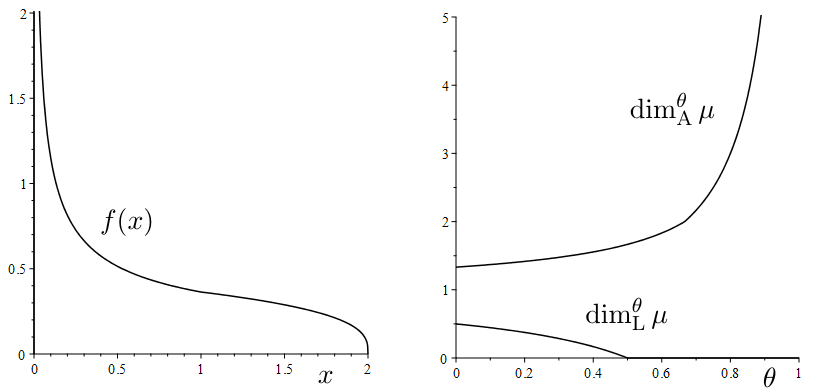}\end{center}
  \caption{A sharp example in the monotonic case, where  $p_1=2, p_2 = 3$ and $C=4/11$.  The density
  is plotted on the left and the Assouad and lower spectra are plotted on the right.  Note that the
Assouad spectrum has a phase transition at $\theta = 2/3$ and the lower spectrum has a phase
transition at $\theta =1/2$.}
  \label{fig:monoexample3}
\end{figure}

\section{A relationship in the opposite direction?}

So far we have proved results of the form: \emph{if a measure is  smooth, then it is also regular}.
In this section we investigate the reverse phenomenon and discover that such a concrete connection
is not possible.  

\subsection{A measure with Assouad dimension $1$ but only $L^1$ smoothness.}

Our first result in this direction shows that the strongest possible assumption on the Assouad
dimension of a measure yields no information about its smoothness.
\begin{theorem}
There exists a compactly supported measure $\mu \in L^1$ with $\dimA \mu = 1$  but which is not in
$L^p$ for $p>1$. 
\end{theorem}

\begin{proof}
Let $p \in (1,1.5]$ and $\mu_p$ be the measure supported on a subset of $[0,1]$ with density
\[
f_p(x)=\begin{cases}2^{-1}\left(2^{k}/k\right)^{1/(p-1)} &x \in B\left(2^{-k},
  \left(k2^{-pk}\right)^{1/(p-1)} \right) \\
0& \textup{otherwise}\end{cases}.
\]
 This is  well-defined since the balls $B\left(2^{-k}, \left(k2^{-pk}\right)^{1/(p-1)} \right) $ are
 pairwise disjoint and a probability measure since
\[
\int f_p(x)dx  = \sum_{k=1}^\infty \left(2^{k}/k\right)^{1/(p-1)}  \left(k2^{-pk}\right)^{1/(p-1)} =
\sum_{k=1}^\infty 2^{-k} = 1.
\]
Moreover,
\[
\int f_p(x)^pdx  = \sum_{k=1}^\infty \left(2^{k}/k\right)^{p/(p-1)}  \left(k2^{-pk}\right)^{1/(p-1)} = \sum_{k=1}^\infty k^{-1} = \infty
\]
and so $\mu_p \notin L^p$.  However, $\mu$ is very regular since $\dimA \mu_p = 1$.  To see this let  $x$ be in the support of $\mu_p$ and $0<r<R$.  We have
\[
\mu_p(B(x,R)) \leq  \sum_{k=m}^\infty 2^{-k} \lesssim 2^{-m},
\]
where $m$ is the largest  integer such that $2^{-m}>x+R$ and
\[
\mu_p(B(x,r))   \gtrsim 2^{-n},
\]
where $n$ is the largest integer such that $2^{-m}<x+r$. In particular,
\[
 \frac{\mu_p(B(x,R))}{\mu_p(B(x,r))} \leq  \left(\frac{2^{-m}}{2^{-n}}\right) \lesssim  \left(\frac{x+R}{x+r}\right) \leq \left(\frac{ R}{ r}\right).
\]
This shows that $\dimA \mu_p = \dimqA \mu_p = \dimAt \mu_p =1$ for all $\theta \in (0,1)$.  Moreover, let $\mu$ be defined by
\[
\mu=\sum_{k=1}^\infty 2^{-k} T_k(\mu_{1+2^{-k}}),
\] 
where $T_k(y) = 2^{-2^k}y+2^{-k}$.  Immediately we see that  $\mu \notin L^p$ for $p>1$. Moreover,
$\dimA \mu =1$, which can be seen by modifying an argument in \cite[Theorem 2.7(2)]{Fraser18}, which
considered the measure $\nu = \sum_k 2^{-k} \delta_{2^{-k}}$ and proved that it has Assouad
dimension 1.  The idea here is that for a given pair of scales $0<r<R$, either the measure $\mu$
looks like the  measure  $\nu$ or like one of the $\mu_p$, due to the super-exponential scaling of
$T_k$. 
\end{proof}

\begin{figure}[ht]
  \begin{center}\includegraphics[width=0.9\textwidth]{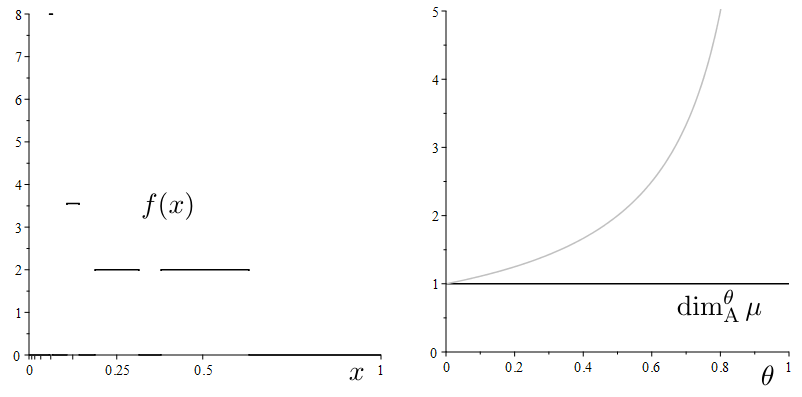}\end{center}
  \caption{The density $f_p$ with $p=1.5$ is plotted on the left and the Assouad spectrum of $\mu_p$ is plotted on the right (constant) along with the upper bound from Theorem \ref{thm:main} for reference (grey).}
  \label{fig:funnyexample}
\end{figure}

\subsection{A   measure with lower dimension 1 but not in $L^{-1}$.}

It is very straightforward to construct a measure with lower  dimension  equal to 1, but which fails
to be in $L^{-1}$.  For example, consider the measure with density  $f(x)=2x$ on $[0,1]$. For any
ball $B(x,r)$ with $0<r<1/2$ we have
\[
\mu(B(x,r)) = \int_{B(x,r)}f = \int_{\max\{0,x-r\}}^{\min\{1,x+r\}}2ydy = \min\{1,x+r\}^2 - \max\{0,x-r\}^2 \approx \begin{cases} 
     r^2 & x<r   \\
      xr& r\leq x 
   \end{cases}
\]
Therefore
\[
R/r \lesssim \frac{\mu(B(x,R)}{\mu(B(x,r)} \lesssim (R/r)^2
\]
with the lower bound attained at $x=1$ and the upper bound attained at $x=0$.  This shows $\dimL \mu
= 1 < \dimA \mu = 2$.  Further $f\in L^\infty([0,1])$ but
$f^{-1}(x) = 1/x$ and so $f\not\in L^{-1}([0,1])$.

\subsubsection{A stronger result for monotonic  densities and further work.}
Assuming $\mu$ has a  monotonic  density we can get an implication that is dual to our main theorem
(letting $\theta \to 0$).

\begin{proposition}
 Suppose $\mu$ is absolutely continuous with monotonic  density $f$ supported on $[0,1]$.   If
 $\dimL\mu \geq  1-1/p$ for some $p>1$, then  $f \in  L^{p'}([0,1])$ for $1 \leq p' < p$.
\end{proposition}
\begin{proof}
  Without loss of generality we may assume that $f(x)>0$ on $(0,1)$ and non-increasing. 
  Let $0<s<1-1/p <1$,   $x=0$ and $0<r<R=1$. Then,
  \[
    C\left(\frac{1}{r}\right)^{s}\leq\frac{\int_{B(x,R)}f(x)dx}{\int_{B(x,r)}f(x)dx} =
    \frac{\int_{0}^{1} f(x)dx}{\int_{0}^{r} f(x)dx} = \frac{1}{\int_{0}^{r} f(x)dx}
  \]
  and so $F(y)  = \int_{0}^{y} f(x)dx \leq C^{-1} y^{s}$ and, since $f$ is non-increasing,
  $f(y)\lesssim y^{s-1} $. Therefore
\[
\| f \|_{p'} \lesssim \int_0^1 x^{p'(s-1)} dx <\infty
\]
provided $p'(1-s)<1$ and therefore $f \in  L^{p'}([0,1])$ for $p' < p$.
\end{proof}

It is easily seen that this cannot hold for arbitrary measures. The balanced Bernoulli measure on
the Cantor middle third set has lower dimension $\log2 / \log 3$ but is not even absolutely
continuous.  We do not know if such a result can be proved for absolutely continuous measures.

\begin{question}
If $\mu \in L^1$ and $\dimL \mu >0$, then is it true that $\mu \in L^p$ for some $p>1$ depending on
$\dimL \mu$?
\end{question}

One might conjecture the following.

\begin{conjecture}\label{thm:mainConj}
If $\mu \in L^1$ and $\dimL \mu >1-1/p$, then   $\mu \in L^p$ for $1 \leq p'<p$.
\end{conjecture}

A proof of this conjecture would require finer detail on the implications of measure decay than we were
able to establish. Consider, for instance, the following straightforward lemma.

\begin{lemma}\label{thm:combiningIntervals}
  Let $\mu$ be an absolutely continuous probability measure supported on $[0,1]$ with density
  $f$. Assume that there exists $C>0$ and $p>1$ such that
  for all $0<r<R<1$ and $x$ in the support of $\mu$,
  \begin{equation}\label{eq:lowerMeasureBound}
    \frac{\mu(B(x,R))}{\mu(B(x,r))} \geq C\left( \frac{R}{r} \right)^{1-1/p}.
  \end{equation}
  Let $I_1$ and $I_2$ be two disjoint intervals of lengths $l_1$ and $l_2$, respectively, that are
  separated by an interval of length $d>0$.
  Then
  \begin{equation}\label{eq:doubleTrouble}
    \mu(B(x_0,R)) \geq C^2  \left( \frac{2R}{l_1+l_2} \right)^{1-1/p}   \mu(I_1\cup I_2)
  \end{equation}
  for $2R\geq l_1+l_2 + 2d$ and $x_0 = a -l_1/(l_1+l_2)+d+(l_1+l_2)/2$, where $a$ is the left-hand
  endpoint of the leftmost of the intervals.
\end{lemma}
\begin{proof}
  Let $x_1$ and $x_2$ be the midpoints of $I_1$ and $I_2$, respectively. Without loss of generality
  we assume $x_1<x_2$. Let $B_1=B(x_1,l_1+d \, l_1/(l_1+l_2))$ and $B_2=B(x_2,l_2+d l_2/(l_1+l_2))$
  and note that $B_1\cap B_2 = \varnothing$ and $\cl(B_1)\cup \cl(B_2)$ is a closed interval containing
  $I_1$ and $I_2$ with midpoint $x_0$. Here $\cl(\cdot)$ denotes the closure.

  Using \eqref{eq:lowerMeasureBound}, we see that
  \begin{align*}
    \mu(B_1\cup B_2) &= C\left( \frac{l_1+2d\frac{l_1}{l_1+l_2}}{l_1} \right)^{1-1/p} \mu(I_1)
    + C\left( \frac{l_2+2d\frac{l_2}{l_1+l_2}}{l_2} \right)^{1-1/p}\mu(I_2)\\
    &=C\left( 1+2d\frac{1}{l_1+l_2} \right)^{1-1/p} (\mu(I_1)+\mu(I_2))\\
    & = C\left( \frac{l_1+l_2+2d}{l_1+l_2} \right)^{1-1/p}\mu(I_1\cup I_2).
  \end{align*}
  Using \eqref{eq:lowerMeasureBound} once more for $2R > l_1+l_2+2d$ we obtain
  \[
    \mu(B(x_0,R)) \geq C\left( \frac{2R}{l_1+l_2+2d} \right)^{1-1/p}\mu(B_1\cup B_2)
    = C^{2}\left( \frac{2R}{l_1+l_2} \right)^{1-1/p} \mu(I_1\cup I_2)
  \]
  as required.
\end{proof}

Observe that \eqref{eq:doubleTrouble} resembles the formula one obtains for an interval $I_0$ of
length $l_1+l_2$ centred at $x_0$ with mass $\mu(I_1\cup I_2)$, namely 
\[\mu(B(x_0,R)) \geq C \left(\frac{2 R}{l_1+l_2}\right)^{1-1/p}\mu (I_1\cup I_2),\]
albeit with an additional factor of $C$.
This suggests that this scheme can be iterated, though the additional constant $C$ as well as the
restriction $R>l_1+l_2+2d$ do not allow this directly.
While we were unable to show this, we suspect that such an iteration can be used to show a statement
such as

\begin{conjecture}\label{thm:allIntervals}
  Let $\mu$ and $f$ be as in Lemma~\ref{thm:combiningIntervals} with the additional assumption that $C\geq 1$.
  Let $\{I_i\}$ be a finite set of pairwise disjoint intervals with $\mu(I_i)\geq\varrho
  \lambda(I_i)$, where $\lambda$ denotes Lebesgue measure. Then,
  \[
    \varrho^{\,p} \lesssim \frac{1}{\sum_{i=1}^N \lambda(I_i)}.
  \]
\end{conjecture}

A special case occurs if all $N$ intervals are of equal length $l$ and equally spaced. Let $x_i$ be the
  midpoint of $I_i$. Then 
  \[
    1=\mu(B(\tfrac12,\tfrac12))
    = \sum_{i=1}^N\mu(B(x_i)) 
    \geq \sum_{i=1}^N C\left(\frac{1/N}{l}\right)^{1-1/p} \mu(I_i) 
    \geq N \varrho\, l \left( \frac{1}{N l} \right)^{1-1/p}
    = \varrho\,(N\,l)^{1/p}
  \]
  and so
  \(
    \varrho^p \leq (N\,l)^{-1}
  \)
  as required.

We can also imagine why this might not hold if  $C<1$. Suppose $C<1$ and place $2^n$ intervals of
length $l$ in a ``Cantor-like'' arrangement, where pairs are separated by a gap of size $\alpha$,
pairs of pairs are separated by $\alpha^2$, and so on.  Then $\alpha>1$ can be picked such that
$C(\alpha l / l)^{1-1/p} = 1$, implying that $f(x)=0$ and the lower dimension condition is still
satisfied.

However, equipped with a statement like Conjecture~\ref{thm:combiningIntervals} one can prove
Conjecture~\ref{thm:mainConj}.

\begin{proof}[Proof of Conjecture~\ref{thm:mainConj} using Conjecture~\ref{thm:allIntervals}]

Let $\fL$ be the Lebesgue points of $f(x)$, i.e.\ the set of points $x$ where $0\leq f(x)<\infty$ and
\[
  g_k(x) = \frac{1}{2} \int_{B(x,2^{-k})} \hspace{-1cm} 2^k|f(y)-f(x)| d\lambda(y) \to 0\qquad
  \text{as}\;k\to\infty.
\]
Define $A_{0} = \{x\in\fL \;:\; f(x)<1\}$ and $A_k = \{x\in\fL\;:\;2^{k}\leq f(x)< 2^{k+1}\}$ for $k\in
\N$.
Note that $\lambda(\fL)=1$ and $\|f\|_p^p = \int_{[0,1]}f^p d\lambda = \sum_{i=0}^\infty
\int_{A_k}f^p d\lambda$ where $\lambda$ is Lebesgue measure.

Fix $\eps>0$ and temporarily fix $k\in \N_0$. We first show that $\int_{A_k} f^{p-\eps}d\lambda
\lesssim 2^{-\eps k}$.
If $\mu(A_k)=0$ we are done. For $k=0$, we get the bound $\int_{A_0}f^{p-\eps}d\lambda\leq 1$.
Hence we can assume $k\geq 1$ and $\mu(A_k)\neq 0$.
Let $0<\delta < \min\{2^{-pk},\mu(A_k)\}$.
Then, by Egoroff's theorem, there exists $B_{\delta}\subseteq\fL$ with
$\lambda(B_\delta)>1-\delta$ such that $g_k(x)\to 0$ uniformly over $x\in B_\delta$.
Let $l$ be large enough such that $2^{-l}<\eps$ and $g_{l'}(x)<\eps$ for all $l'\geq l$ and $x\in
B_\delta$.
Now let $I= \bigcup_{x\in B_\delta\cap A_k}B(x,2^{-l})$ and note that $I$ is a finite collection
of intervals $\{I_i\}$ with $\lambda(I_i)\geq 2^{-l}$.  We
obtain
\[
  \int_{I}f^{p-\eps}d\lambda \leq \sum_{i=1}^{\#\{I_i\}}\lambda(I_i)\left( 2^{k+1}+\delta
  \right)^{p-\eps}
  \leq \sum_{i=1}^{\#\{I_i\}}\lambda(I_i) 2^{(k+2)(p-\eps)}.
\]
and so, using Conjecture~\ref{thm:allIntervals},
\[
  \int_{A_k}f^{p-\eps}d\lambda \leq \int_{A_k\cap B_\delta} \hspace{-2em}f^{p-\delta}d\lambda
  +\int_{A_k\cap ([0,1]\setminus B_\delta)}\hspace{-4em}f^{p-\eps}d\lambda
  \leq 2^{2(p-\eps)}2^{k(p-\eps)}\sum_{i=1}^{\#\{I_i\}}\lambda(I_i)
  \;+\;\lambda([0,1]\setminus B_\delta) 2^{(p-\eps)k}
  \lesssim 2^{-\eps k}.
\]
Finally,
\[
  \|f\|^{p-\eps}_{p-\eps} d\lambda = \sum_{k=0}^\infty \int_{A_k}f^{p-\eps}d\lambda \leq
  \sum_{k=0}^\infty 2^{-\eps k}<\infty
\]
and so $\mu\in L^{p-\eps}$ and letting $\eps \to 0$ proves Conjecture~\ref{thm:mainConj}. 
\end{proof}

\section{Absolute continuity with general reference measures.}

For completeness we include the following result which considers  absolute continuity with respect
to general reference measures, the proof of which is almost identical to Theorem \ref{thm:main}

\begin{theorem} \label{generalgeneral}
 Let $\nu$ be a  measure supported on a non-empty compact set  $X \subseteq \R^d$ and suppose  $s,t\geq 0$ are such that for all $x \in X$ and $r>0$
  \begin{equation} \label{regdi}
r^s \lesssim \mu(B(x,r)) \lesssim r^t.
  \end{equation}
Suppose  $\mu$ is  a measure that is absolutely continuous with respect to $\nu$
 and suppose $p_1, p_2 \in [1,\infty)$ are such that
 \[
 \lVert f
    \rVert_{p_1} =   \int_X f^{p_1} d \nu < \infty \quad\text{and}\quad
  \lVert f
    \rVert_{-p_2} = \int_X f^{-p_2} d\nu < \infty.
 \]
Then
  \begin{equation*}
    \dimAt \mu \leq \frac{s-\theta t}{1-\theta} +\frac{p_1s+\theta p_2t}{p_1 p_2 (1-\theta)}
    \quad \text{and} \quad 
    \dimLt \mu \geq \frac{t-\theta s}{1-\theta}-\frac{\theta p_1s + p_2t}{p_1 p_2 (1-\theta)} .
  \end{equation*}  
\end{theorem}

\begin{proof}
 Fix $\theta\in(0,1)$ and  $0<r=R^{1/\theta} \leq 1$.  Write $q_1\in (1,\infty)$ for the H\"older conjugate of $p_1$.  Then, by H\"older's inequality
  \begin{align*}
    \frac{\mu(B(x,R))}{\mu(B(x,r))} \leq \frac{\lVert f
    \rVert_{p_1}\lVert\chi_{B(x,R)}\rVert_{q_1}}{\lVert f \rVert_{-p_2}\lVert
    \chi_{B(x,r)}\rVert_{p_2/(1+p_2)}} 
    &\lesssim     \frac{\left(\int \chi_{B(x,R)}^{q_1}d\nu\right)^{1/q_1}}{\left(\int
    \chi_{B(x,r)}^{p_2/(1+p_2)}d\nu\right)^{1+1/p_2}}\\
    &\lesssim  \dfrac{R^{t(1-1/p_1)}}{r^{s(1+1/p_2)}}  \quad \text{(by \eqref{regdi})}\\
    &=   \left(\frac{R}{r}\right)^{\frac{t(1-1/p_1)-s(1+1/p_2)/\theta}{1-1/\theta}}.
  \end{align*}
Therefore,
\begin{equation*}
  \dimAt\mu \leq \frac{t(1-1/p_1)-s(1+1/p_2)/\theta}{1-1/\theta}
  =\frac{s-\theta t}{1-\theta} +\frac{p_1s+\theta p_2t}{p_1 p_2 (1-\theta)},
\end{equation*}
as required.  The estimate for the lower spectrum is similar and omitted, see the proof of Theorem \ref{thm:main}.
\end{proof}

Note  that, for all $\varepsilon>0$, we can always choose $s$ and $t$ in the statement of Theorem \ref{generalgeneral}  satisfying
\[
\dimL \nu -\varepsilon \leq t \leq s  \leq \dimA \mu + \varepsilon
\]
but better choices are sometimes possible.

\section*{Acknowledgements.}
This work was started while ST was visiting JMF at the University of St Andrews in April 2019. 
ST wishes to thank St Andrews for the hospitality during his visit.

\end{document}